\documentclass[a4paper,10pt]{amsart}
\usepackage{amssymb}
\usepackage{array}
\usepackage{graphicx}
\usepackage[english]{babel}
\usepackage[utf8]{inputenc}
\usepackage{amsmath}
\usepackage{amsthm}
\usepackage{units}
\usepackage{dsfont}
\usepackage[arrow, matrix, curve]{xy}
\newcommand{\tu}[1]{\textup{#1}}

\newcommand{\Abb}[4]{\left\{ \begin{array}{ccc}
                               #1 & \rightarrow &#2\\
			       #3 &\mapsto &#4
                               \end{array}\right.}
\newcommand{\supp}{\operatorname{supp}}

\newcommand{\N}{\mathbb{N}}
\newcommand{\R}{\mathbb R}

\newcommand{\C}{\mathbb{C}}

\newcommand{\WF}{\operatorname{WF}}

\theoremstyle{plain}
\newtheorem{Satz}{Satz}
\newtheorem{lem}[Satz]{Lemma}
\newtheorem{prop}[Satz]{Proposition}
\newtheorem{thm}[Satz]{Theorem}

\theoremstyle{definition}

\newtheorem{Def}{Definition}

\theoremstyle{remark}
\newtheorem{rem}{Remark}

\title[Support of Pollicott-Ruelle states]{On the support of Pollicott-Ruelle resonant states for Anosov flows}
  
\author{Tobias Weich}
\address{Institut für Mathematik, Universitat Paderborn, Paderborn, Germany}
\email{weich@math.uni-paderborn.de}

\begin{document}
\begin{abstract}
We show that all generalized Pollicott-Ruelle resonant states of a topologically transitiv 
$C^\infty$ Anosov flow with an arbitrary $C^\infty$ potential, have full support.
\end{abstract}
\maketitle
\section{Introduction}
\label{sec:intro}
Let $(M,g)$ be a smooth, compact Riemannian manifold without boundary.
Let $X\in C^\infty(M,TM)$ be a smooth Anosov vector field and let us denote 
by $\varphi_t$ the flow on $M$ generated by $X$ and let $\mathcal V\in C^\infty(M;\C)$
be a smooth potential function. Then we can define the following 
differential operator 
\begin{equation*}
\mathbf P:= \frac{1}{i} X + \mathcal V: C^\infty(M)\to C^\infty(M).
\end{equation*}
It is a well established approach to study the dynamical properties of  
Anosov flows by the discrete spectrum of the operator $\mathbf P$, the so called 
Pollicott-Ruelle resonances. The fact that for volume preserving flows and 
real valued potentials, the operator
$\mathbf P$ is an unbounded, essentially self adjoint operator on $L^2(M)$ 
might suggest, that $\mathbf P$ has good spectral properties on $L^2(M)$.
However, due to the lack of ellipticity $\mathbf P$ has mainly continuous spectrum 
which carries little information on the dynamics of the flow. A very important 
progress was thus to construct Banach spaces  \cite{Liv04,BL07} or Hilbert spaces
\cite{FSj11,arXDyZw13} for Anosov flows in which the operator $\mathbf P$ has discrete 
spectrum in a sufficiently large region (see also \cite{Rug92,  Kit99, BKL02, GL06,BT07}
for analogous results for Anosov diffeomorphism). More precisely, it has been shown
that there is a family of Hilbert spaces $H_{sG}$ parametrized by $s>0$ 
such that for any $C_0>0$  and for sufficiantly large $s$
the operator $\mathbf P$ acting on $H_{sG}$ has discrete spectrum in the region 
$\{\tu{Im}\lambda >-C_0\}$. This discrete 
spectrum is known to be intrinsic to the Anosov flow together with the potential 
function and does not depend on the sufficiantly large parameter $s$ (see 
Section~\ref{sec:micro_ruelle} for a more precise statement).
Accordingly we call $\lambda_0 \in \C$ a \emph{Pollicott-Ruelle resonance} 
if it is an eigenvalue of $\mathbf P$ on $H_{sG}$ for sufficiently large $s$
and we call $\ker_{H_{sG}}(\mathbf P-\lambda_0)$ the space of \emph{Pollicott-Ruelle resonant states}.
As $\mathcal P$ is not anymore a normal operator on $H_{sG}$, there might exist also 
exist finite dimensional Jordan blocks and given a Ruelle resonance $\lambda_0$
we denote by $J(\lambda_0)$ the maximal size of a corresponding Jordan block and
we call $\ker_{H_{sG}}(\mathbf P-\lambda_0)^{J(\lambda_0)}$ the space of \emph{generalized
Pollicott-Ruelle resonant states}.

The interest in these Pollicott-Ruelle resonances and their resonant states arises
from the fact that they govern the decay of correlations. In order to illustrate 
this in more detail, let us for the moment assume that the Anosov flow is contact,
which includes for example all geodesic flows on compact manifolds 
with negative sectional curvature. If we denote by $dw$ the invariant contact 
volume form and take two smooth functions $f,g\in C^\infty(M)$, 
then one is interested in the behavior of the correlation function
\[
 C_{f,g}(t):=\int_M f(m)\varphi_t^*g(m) dw.
\]
In this case the following expansion of the correlation function has been established 
in \cite[Corollary 1.2]{Tsu10} (see also \cite[Corollary 5]{NZ13})
\begin{equation}
 \label{eq:cor_expansion}
 C_{f,g}(t) = \int_M f\,dw~ \int_M g\,dw + \sum_{j=1}^J\sum_{k=1}^{K_j} t^k e^{-it\mu_j} 
 u_{j,k}(f)v_{j,k}(g) +\mathcal O(e^{-\gamma t}).
\end{equation}
Here $\mu_j$ are the Pollicott-Ruelle resonances for a vanisihing potential
($\mathcal V=0$) which are lying in a strip 
$0>\tu{Im}(\mu)>-\gamma$. Furthermore $u_{j,1}, v_{j,1}$ are the corresponding left and 
right generalized eigenstates and the sum over $k$ arises from the possible existence of
Jordan blocks of size $K_j$. As $\tu{Im}(\mu_j)<0$ all terms except the first 
one vanish exponentially in the limit of large times. This effect is interpreted 
as an exponential convergence towards equilibrium, known as exponential 
mixing and has first been established in the seminal works of Dolgopyat 
\cite{Dol98} and Liverani \cite{Liv04}. The way towards equilibrium is governed 
by the sub leading resonances and resonant states. While the resonances determine 
the possible decay modes, the resonant states determine the coefficients. If for example
there would be a simple Pollicott-Ruelle resonance $\mu_j$ with a resonant state, 
that vanishes on an open set $U\subset M$, then for all observables 
$f\in C^\infty_c(U)$ this resonance would not appear in the correlation expansion
(\ref{eq:cor_expansion}). The main theorem of this article states, that
this can not be the case. 
\begin{thm}\label{thm:support}
 Let $\varphi_t$ be a topologically transitive $C^\infty$ Anosov flow. Let $\lambda_0$ be a 
 generalized Pollicott-Ruelle resonance and 
 $u\in\ker_{H_{sG}}(\mathbf P-\lambda_0)^{J(\lambda_0)}\setminus\{0\}$
 then $\supp u = M$.
\end{thm}
\begin{rem}
Note that contrary to the correlation expansion (\ref{eq:cor_expansion})
which we mentioned as a motivation, our result for the resonant states 
holds for a general $C^\infty$ Anosov flow with no assumption 
on contact structures, smooth invariant measures, or even mixing properties. 
Taking a suspension of a topological transitive Anosov diffeomorphism, the theorem 
directly implies an analogous result for the resonant states of the Anosov 
diffeomorphism. 
\end{rem}
\begin{rem}
To our knowledge there has up to now little been known about the structure of 
Pollicott-Ruelle eigenstates. However in the proof of \cite[Theorem 5.1]{GL08}
Gou\"ezel and Liverani show an analogouse result for the peripheral spectrum
of topologically mixing hyperbolic maps. In this case, there is however only 
one simple peripheral eigenvalue so the statement applies only to this one resonant state
which corresponds to the equilibrium measure. Furthermore for the particular case 
of a geodesic flows on manifolds of constant negative curvature, Dyatlov, Faure
and Guillarmou \cite{DFG15} proved an explicit relation between Pollicott-Ruelle
resonant states and Laplace eigenstates. Using this relation allows to transfer 
well established support properties of Laplace eigenstates to support properties 
of Pollicott-Ruelle eigenstates.
\end{rem}
The article is organized as follows: In Section~\ref{sec:notation} we introduce
the basic definitions and recall some known facts about Anosov flows and the 
microlocal approach to Pollicott-Ruelle resonances. Section~\ref{sec:support} is
devoted to the proof of the main Theorem~\ref{thm:support}.

\emph{Acknowledgements:}
I would like to thank Heiko Gimperlein for enlightening discussions about the 
microlocal techniques in this proof. Furthermore I am grateful to 
Sebastién Gou\"ezel, Stéphane Nonnenmacher, Viviane Baladi and Gabriel Rivière
for the hints concerning the dynamical systems literature and in particular for the explications
about the smoothness of the conditional measures (c.f. Theorem~\ref{thm:anosov_fubini})
by the first two of them. 
Last but not least, I thank Colin Guillarmou and Joachim Hilgert for stimulating
discussions, from which this work emerged as well as for their detailed feed-back 
on an earlier version of the manuscript. This work
has been supported by the grant DFG HI 412 12-1. 

\section{Notation and Preliminaries}\label{sec:notation}
\subsection{Anosov flows}
Let $(M,g)$ be a smooth, compact Riemannian manifold without boundary and let us 
denote the Riemannian distance between two points $m,\tilde m\in M$ by 
$d(m,\tilde m)$. A smooth vector field $X\in C^\infty(M,TM)$, respectively the flow
$\varphi_t$ on $M$ generated by $X$, are called \emph{Anosov} if there is a 
direct sum decomposition of each tangent space
\[
 T_m M = E_0(m)\oplus E_s(m)\oplus E_u(m)
\]
which depends continuously on the base point $m$ and which is invariant under
 $d\varphi_t$. Furthermore the decomposition has to fulfill $E_0(m) =\R X(m)$ 
and there has to exist some fixed $C$ and $\theta>0$ with
\begin{equation}
 \begin{split}
  \|d\varphi_t(m)v\|_{T_{\varphi_t(m)} M} \leq C e^{-\theta|t|}\|v\|_{T_{m} M}
  &~~\forall v\in E_u(m),~~t<0\\
  \|d\varphi_t(m)v\|_{T_{\varphi_t(m)} M} \leq C e^{-\theta|t|}\|v\|_{T_{m} M} &~~ \forall v\in E_s(m),~~t>0.
 \end{split}
\end{equation}
An Anosov flow is said to be \emph{topological transitive} if there exists a 
dense orbit. 

It is known, that the bundles $E_0\oplus E_s$ and $E_0\oplus E_u$ are uniquely integrable
\cite{Ano67, HPS70} and given a point $m\in M$ we denote by $W^{wu}(m)$ and $W^{ws}(m)$ 
the integral manifolds through the point $m$ which are smooth 
immersed submanifolds of $M$. They are called \emph{weak stable}
and \emph{weak unstable} manifolds through $m$. Also the bundles $E_s$ and $E_u$ are uniquely
integrable. The corresponding integral manifolds will be denoted by $W^s(m), W^u(m)$
and will be called \emph{strong stable} and \emph{strong unstable} manifolds. 
For further reference let us denote by $n_{ws},n_{wu},n_s,n_u\in \N$ the dimension 
of the manifolds $W^{ws}(m), W^{wu}(m), W^s(m)$ and $W^u(m)$, respectively. Let us
recall the following well established result on the continuous dependence of the invariant 
manifolds on the base point.
\begin{thm}[c.f.~{\cite[Theorem 3.2]{HP70}}]\label{thm:cont_depend_inv_mflds}
 $W^{ws}$ is a family of $n_{ws}$-dimensional immersed $C^\infty$-submanifolds that 
 depends continuously on the base point w.r.t. the $C^\infty$-topology. More precisely,
 this means that for any $m$ there is a neighborhood $m\in U\subset M$ as well as a
 continuous map $g: U \to C^\infty(D^{n_{ws}},M)$ such that for any $\tilde m\in U,$
 $g(\tilde m)$ is a $C^\infty$-diffeomorphism of the $n_{ws}$ dimensional 
 disk $D^{n_{ws}}$ onto a neighborhood of $\tilde m$ in $W^{ws}(\tilde m)$.
 
 The analogous result holds for $W^{wu}, W^s, W^u$. 
\end{thm}

\begin{rem}
 The stable and unstable foliations are known to be even Hölder continuous
 in general, and under certain assumptions such as curvature pinching even regularity
 $C^1$ or better can be obtained (c.f. \cite[Section 2.3]{Has02}). We will, 
 however, not explicitly need these refined regularity estimates in the
 sequel.
\end{rem}

Given two points $y,z\in W^{ws}(m)$ we denote by $d^{ws}(y,z)$ the metric distance on 
$W^{ws}(m)$ coming from the metric inherited from the Riemannian metric on $M$. 
In the same way we can define $d^{wu}, d^s$ and $d^u$.
We can now define the following different open balls of radius $r>0$.
\begin{align*}
B_r(m)&:=\{\tilde m\in M: d(m,\tilde m)<r \}\\
B_r^{wu}(m)&:=\{\tilde m\in W^{wu}(m): d^{wu}(m,\tilde m)<r \}\\
B_r^{ws}(m)&:=\{\tilde m\in W^{ws}(m): d^{ws}(m,\tilde m)<r \}\\
B_r^u(m)&:=\{\tilde m\in W^u(m): d^{u}(m,\tilde m)<r \}\\
B_r^s(m)&:=\{\tilde m\in W^s(m): d^{s}(m,\tilde m)<r \}.
\end{align*}
\begin{thm}[Product Neighborhood Theorem]\label{thm:product_neighbour}
There exists $\delta_0 >0$ independent of $m\in M$ such that for any $m\in M$
and $\delta\leq\delta_0$ the following maps 
\begin{align*}
H^{ws,u}_m:&\Abb{B^{ws}_\delta(m)\times B^u_\delta(m)}{M}{(x,y)}{B^u_{2\delta}(x)\cap B^{ws}_{2\delta}(y)}\\
H^{wu,s}_m:&\Abb{B^{wu}_\delta(m)\times B^s_\delta(m)}{M}{(x,y)}{B^s_{2\delta}(x)\cap B^{wu}_{2\delta}(y)}
\end{align*}
are unambiguously defined, injective and homeomorphisms onto their images.
\end{thm}
For a proof see \cite{PS70} (see also \cite{Pla72} for a more detailed statement). 
Given a point $m\in M$ and $0<\alpha,\beta<\delta_0$ we define the open product 
neighborhoods of $m$
\[
 \mathcal{PN}^{ws,u}_{\alpha,\beta}(m):= H^{ws,u}_m(B^{ws}_\alpha(m)\times B^u_\beta(m)) \subset M
 \]
and
\[
\mathcal{PN}^{wu,s}_{\alpha,\beta}(m):= H^{wu,s}_m(B^{wu}_\alpha(m)\times B^s_\beta(m))\subset M.
\]
We will call such product neighborhoods also rectangular neighborhoods. These
product neighborhoods have the important property that the defining homeomorphisms
imply foliations by local invariant manifolds. In the sequel we will only need the
foliations of $(wu,s)$-rectangles which are given by
\[
 \mathcal{PN}^{wu,s}_{\alpha,\beta}(m) = \bigcup_{x\in B^{wu}_\alpha(m)} \mathcal S_x = \bigcup_{y\in B^{s}_\beta(m)} \mathcal U_y
\]
where $\mathcal S_x := H^{wu,s}_m(\{x\}\times B^s_\beta(m))$ are the local strong stable 
leaves and $\mathcal U_y := H^{wu,s}_m(B^{wu}_\alpha(m)\times \{y\})$ the local weak unstable leaves.
Note that from Theorem~\ref{thm:cont_depend_inv_mflds}, as well the strong stable 
leaves $\mathcal S_x$ as the weak unstable leaves $\mathcal U_y$ are smooth 
submanifolds of the product neighborhood. Still the foliations are no smooth foliations
as the trivialization map $H^{wu,s}_m$ is only continuous, which reflects the fact 
from Theorem~\ref{thm:cont_depend_inv_mflds} that the local invariant manifolds 
$\mathcal S_x$ ($\mathcal U_y$ respectively), depend only continuously on their 
basepoints $x$ ($y$ respectively). As $M$ is a Riemannian manifold, the leaves 
$\mathcal S_x$ ($\mathcal U_y$ respectively) inherit a Riemannian metric and thus also a 
Lebesgue measure which we denote by $dm_{\mathcal S_x}$ ($dm_{\mathcal U_y}$ respectively).

One way to introduce Sinai-Ruelle-Bowen (SRB) measures is to demand, that they 
are absolute continuous w.r.t. the weak unstable foliation.
\begin{Def}[SRB measure]
 A measure $\mu$ on $M$ which is invariant under the Anosov flow $\varphi_{t}$
 is called a \emph{SRB-measure} if for any product neighborhood $\mathcal{PN}^{wu,s}_{\alpha,\beta}(m)$
 and for all $y\in B^s_\beta(m)$ there are positive, measureable functions $\tau_y$
 on $\mathcal U_y$ as well as a measure $d\sigma$ on $B^s_\beta(m)$ such that the 
 restriction of $\mu$ to $\mathcal{PN}^{wu,s}_{\alpha,\beta}(m)$ is given by
 \[
\mu_{|\mathcal{PN}^{wu,s}_{\alpha,\beta}(m)} = \int_{B^s_\beta(m)} (\tau_y dm_{\mathcal U_y}) d\sigma(y).
 \]
 More precisely this equality means, that for any $f\in C_c(\mathcal{PN}^{wu,s}_{\alpha,\beta}(m))$
 \[
  \int f d\mu = \int_{B^s_\beta(m)}\left(\int_{\mathcal U_y} f_{|\mathcal U_y}(z) \tau_y(z) dm_{\mathcal U_y}(z)\right)d\sigma(y).
 \]
\end{Def}
\begin{rem}
 There are other equivalent possibilities to define SRB measure in terms of 
 metric entropy and ergodic averages, see e.g. \cite{You02} for an overview. 
\end{rem}
\begin{thm}
For any Anosov flow there exists a SRB measure. If the Anosov 
flow is topological transitive, the SRB measure is unique and the Anosov flow is 
ergodic w.r.t. the SRB measure.
\end{thm}
For an original proof of this result in the context of Axiom A diffeomorphism
see \cite{Bow70} and for generalizations to flows \cite{Bow73}.

\subsection{Microlocal approach to Pollicott-Ruelle resonances for Anosov flows}
\label{sec:micro_ruelle}
Let us collect some facts about the spectral theory of Pollicott-Ruelle resonances.
Suitable Banach spaces, in which the differential operator $\mathbf P$, 
has discrete spectrum, have first been introduced by 
Liverani \cite{Liv04}. In this article we will use the microlocal approach to 
transfer operators, which has been introduced by Faure, Roy and Sjöstrand in a 
series of papers \cite{FR06, FRS08, FSj11}. The following result for Anosov flows 
has originally been shown by Faure and Sjöstrand \cite[Theorem 1.4]{FSj11}. 
The reader might also be interested in \cite[Proposition 3.2]{arXDyZw13} where 
an alternative proof, using different microlocal techniques, is given. While these
two publications only mention the case of zero potential a more general 
statement including arbitrary smooth potentials can be found in \cite{arXDG14}.
\begin{thm}\label{thm:fredholm_property}
For any $C_0>0$ there is an  $s>0$ and Hilbert spaces $H_{sG}$ and $D_{sG}$  
such that in the region $\{\tu{Im}\lambda >-C_0\}$
the operator $\mathbf P-\lambda:D_{sG}\to H_{sG}$ is Fredholm of index 0.

Here the Hilbert space $H_{sG}$ is an anisotropic Sobolev space 
 that fulfills the relations
\[
C^\infty(M)\subset H^s(M)\subset H_{sG}\subset H^{-s}(M)\subset \mathcal D'(M),
\]
where $H^s(M)$ denotes the ordinary Sobolev space. The Hilbert spaces $D_{sG}$ 
can be defined by considering the operator $\mathbf P$ acting on 
$\mathcal D'(M)$ and we set 
\[
D_{sG}:=\{u\in H_{sG}: \mathbf P u\in H_{sG}\}
\]
which is a Hilbert space with the norm $\|u\|_{D_{sG}} := \|u\|_{H_{sG}} + \|\mathbf P u\|_{H_{sG}}$
\end{thm}
Since $(\mathbf P -\lambda)$ is invertible for $\tu{Re}(\lambda)$ large enough
(see e.g. \cite[Lemma 3.3]{FSj11} or \cite[Prop 3.1]{arXDyZw13}), analytic Fredholm theory 
implies that the resolvent
\[
\mathbf R(\lambda):=(\mathbf P-\lambda)^ {-1}:H_{sG} \to H_{sG}
\]
has a meromorphic continuation to $\{\tu{Im}\lambda >-C_0\}$ and the 
\emph{Pollicott-Ruelle resonances}
are defined to be the poles of the meromorphic continuation. Given a Pollicott-Ruelle
resonance $\lambda_0$ the space $\ker_{H_{sG}}(\mathbf P-\lambda_0)\setminus \{0\}$ is
nonempty and its elements are called the corresponding 
\emph{Pollicott-Ruelle resonant states}. Furthermore there is an integer 
$J(\lambda_0)\geq 1$ determining the maximal size of a Jordan block with spectral
value $\lambda_0$, i.e. we have for all $k\geq J(\lambda_0)$ 
$\ker_{H_{sG}}(\mathbf P-\lambda_0)^k = \ker_{H_{sG}}(\mathbf P-\lambda_0)^{J(\lambda_0)}$ and we 
call the space $\ker_{H_{sG}}(\mathbf P-\lambda_0)^{J(\lambda_0)}$ the space of 
\emph{generalized Pollicott-Ruelle resonant states}. Note that the
fact that $C^\infty(M)$ is densely contained in all $H_{sG}$ together with the uniqueness
of meromorphic continuation imply that neither the position of the resonance
$\lambda_0$ nor the spaces $\ker_{H_sG}(\mathbf P-\lambda_0)$ 
and $\ker_{H_sG}(\mathbf P-\lambda_0)^{J(\lambda_0)}$ depend on the 
parameter $s$, so they are intrinsic objects of the Anosov flow (c.f. 
\cite[Theorem 1.5]{FSj11}). 

A central ingredient in the proof of Theorem~\ref{thm:fredholm_property} 
and the construction of the anisotropic Sobolev spaces is to 
study the symplectic lift of the flow action to $T^*M$. Faure and Sjöstrand 
therefore introduce the decomposition
\[
 T^*_mM = E_0^*(m)\oplus E_s^*(m)\oplus E_u^*(m)
\]
which is defined such that 
\begin{align*}
 E^*_0(m)(E_s(m)\oplus E_u(m))&=0,\\
 E^*_u(m)(E_0(m)\oplus E_u(m))&=0,\\
 E^*_s(m)(E_0(m)\oplus E_s(m))&=0.
\end{align*}
Note that if one identifies $T_mM$ with $T^*_mM$ via the Riemannian metric this
definition seems counterintuitive, as $E_s^*(m)$ is identified with $E_u(m)$ and 
vice versa. The reason for this choice in \cite{FSj11} is that it is natural 
from the point of view of the symplectic lift of $\varphi_t$ to $T^*M$ (c.f. 
\cite[eq. (1.13)]{FSj11}). We will not need further details on the construction
of the anisotropic Sobolev spaces and refer to the excellent expositions in 
\cite{FSj11, arXDyZw13}. However, we will crucially use the following 
consequence of these constructions on the microlocal regularity of the 
Pollicott-Ruelle resonant states \cite[Theorem 1.7]{FSj11}
\begin{equation}
 \label{eq:wavefront_cond}
 \ker_{H_{sG}}(\mathbf P-\lambda_0) \subset \mathcal D'_{E^*_u}(M),
\end{equation}
where $D'_{\Gamma}(M)$ denotes the space of distributions on $M$ whose wavefront sets
are contained in some closed cone $\Gamma\subset T^*M$ (see \cite[Lemma 8.2.1]{HoeI}
for further details).

\section{On the support of Pollicott-Ruelle eigenstates}\label{sec:support}
In this section, we will prove Theorem~\ref{thm:support}. Let us start
by giving a brief outline of the proof: As a first step, we will introduce suitable 
charts that are adapted to the foliation of $M$ into stable and unstable manifolds.
Then we study some particular distributions, denoted by $\rho_A$, which are defined
in these charts and which correspond to characteristic functions of tubes 
along the strong stable foliation. Proposition \ref{prop:foruier_estimate} then 
gives estimates on the decay of the Fourier transformation of $\rho_A$. As a consequence,
one deduces that the wavefront set of the distributions $\rho_A$ are bounded away from 
$E_u^*$, thus they can be paired (or multiplied) with Pollicott-Ruelle resonant
states. Using the Fourier estimates in Proposition~\ref{prop:foruier_estimate}
we then show, that any Pollicott-Ruelle resonant state that does not 
have full support, has to be identically zero. Therefore we only use two
properties of the resonant states: the flow invariance
of their support and their microlocal regularity. A major technical difficulty in
the proof comes from the fact that the stable and unstable foliations for general 
Anosov flows are only Hölder regular.

We call a finite open cover $M=\bigcup_{i=1}^N R_i$ a cover of product neighborhoods
if there are $m_i\in M$, $r_i^{wu}, r_i^s > 0$ such that 
$R_i = \mathcal{PN}_{r_i^{wu},r_i^s}^{wu,s}(m_i)$ (c.f. Theorem~\ref{thm:product_neighbour}).
In order to simplify the notation 
 we write $X_i:=B_{r_i^{wu}}^{wu}(m_i)$, $Y_i:=B_{r_i^s}^s(m_i)$ and denote 
 the homeomorphism of Theorem~\ref{thm:product_neighbour} by
 $h_i:= H_{m_i}^{wu,s}:X_i\times Y_i\to R_i$.
recall that $h_i(X_i\times\{m_i\}) =B_{r_i^{wu}}^{wu}(m_i)\subset R_i$ is a smooth 
submanifold. 

Now let us construct a smooth atlas of $M$ by choosing $C^\infty$-charts 
$\kappa_i: R_i \to V_i \subset \R^n=\R^{n_{wu}}\times \R^{n_s}$ for our open 
cover of product neighborhoods such that
\begin{equation}
 \label{eq:chart_properties}
 \begin{array}{rcl}
 \kappa_i(m_i)&=&0\in \R^n, \\
 \kappa_i(B_{r_i^{wu}}^{wu}(m_i)) &=& (\R^{n_{wu}}\times\{0\})\cap V_i,\\
 \kappa_i(B_{r_i^{s}}^{s}(m_i)) &= &(\{0\}\times \R^{n_{s}})\cap V_i.
 \end{array}
\end{equation}
\begin{figure}
\centering
        \includegraphics[width=1\textwidth]{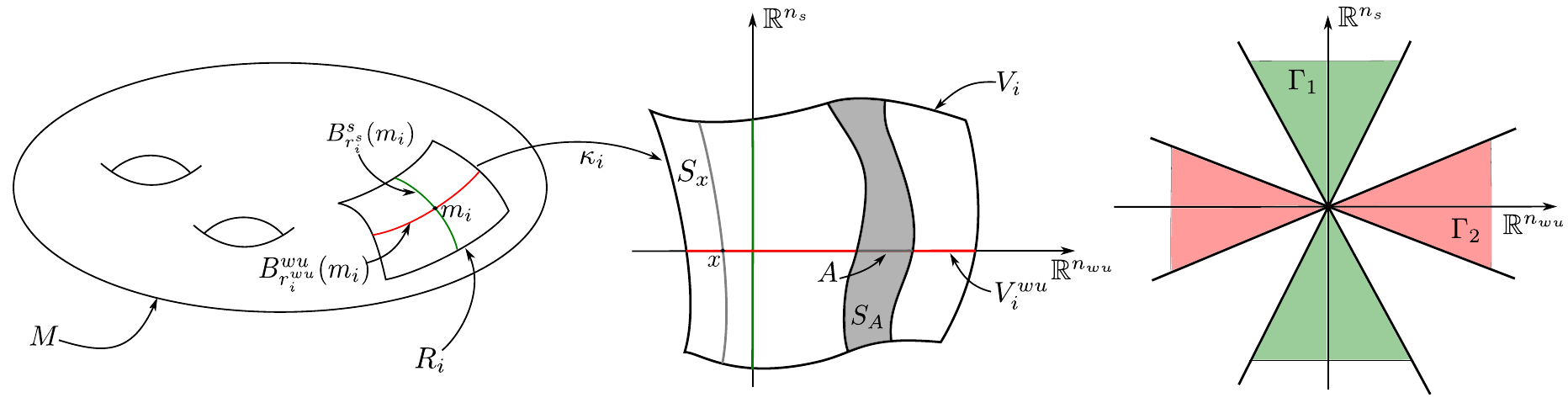}
\caption{Sketch visualizing the properties of the charts $\kappa_i$}
\label{fig:kappa}
\end{figure}
We can then define the open set $V_i^{wu}:=\{x\in \R^{n_{wu}}: (x,0)\in V_i\}$ and
directly see that 
\[
h_i^{wu}: x\in X_i\mapsto \kappa_i\circ h_i(x,m_i) \in V_i^{wu}\times\{0\}\cong V_i^{wu}
\]
defines a homeomorphism between $X_i$ and $V_i^{wu}$. Now for any Borel set 
$A \subset V_i^{wu}$ we define
\[
S_A:= \kappa_i\circ h_i((h_i^{wu})^{-1}(A)\times Y_i) \subset V_i
\]
which is again a Borel subset. Note that for a point $x\in V_i^{wu}$ the
sets $S_x:=S_{\{x\}}$ are exactly the image of the local strong stable manifolds 
under the coordinate chart $\kappa_i$
\[
  S_x = \kappa_i(B_{r_i^s}^s((h^{wu}_i)^{-1}(x))).
\]
According to Theorem~\ref{thm:cont_depend_inv_mflds} the 
$S_x\subset V_i$ are $C^\infty$-submanifolds depending continuously
on the basepoint $x$ w.r.t.~the $C^\infty$ topology. Furthermore, the definition of $S_A$ 
allows us to associate to any Borel set $A\subset V_i^{wu}$ the 
distribution $\rho_A \in \mathcal D'(V_i)$ by setting for $\psi\in C_c^\infty(V_i)$
\begin{equation}\label{eq:rho_A}
 \rho_A[\psi]:= \int_{V_i}\mathds{1}_{S_A}(z)\psi(z) dz.
\end{equation}
Given a chart $\kappa_i$ we write $T^*V_i=V_i\times \R^n$ and we define $\Gamma_1,
\Gamma_2\subset \R^n$ to be the minimal closed cones such that
\begin{equation}
 \label{eq:gamma_1_properties}
 \forall~ m\in R_i,\eta\in E^*_u(m):~~(\kappa_i^{-1})^*(\eta) \in V_i\times \Gamma_1\subset T^*V_i 
\end{equation}
and
\begin{equation}
 \label{eq:gamma_2_properties}
 \forall~ x \in V_i^{wu}, (s,\xi)\in N^*S_x\subset V_i\times \R^n:~~ \xi\in \Gamma_2. 
\end{equation}
Note that for $\eta\in E^*_u(m_i)$ one has $(\kappa_i^{-1})^*(\eta)\in \{0\}\times(\{0\}\times \R^{n_s})$
and that on the other side $N^*S_0 = \{0\}\times(\R^{n_{wu}}\times\{0\})$ are complementary
cones. Thus from the continuity of the foliations and after a possible refinement 
of the charts we can assume that
\begin{equation}\label{eq:un_stable_cones}
\Gamma_1\cap \Gamma_2 =\{0\}
 \end{equation}

We now want to study the Fourier transform of $\rho_A$:

\begin{prop}\label{prop:foruier_estimate}
 Let $\chi\in C_c^\infty(V_i)$ be an arbitrary cutoff function, then there is a
 constant $C>0$ such that for any Borel set $A\subset V_i^{wu}$, 
 \begin{equation}\label{eq:rho_FT_decay1}
  |\widehat{\chi\cdot\rho_A}(\xi)|\leq C \tu{vol}_{n_{wu}}(A)
 \end{equation}
uniformly in $\xi\in \R^n$. Here $\tu{vol}_{n_{wu}}(A)$ is the
$n_{wu}$-dimensional euclidean volume of $A\subset V_i^{wu}\subset \R^{n_{wu}}$.

Furthermore, fix a closed cone $\Omega\subset \R^n$ such that 
$\Omega\cap \Gamma_2 = \{0\}$. Then for all $N\in \N$ there is a
$C_N>0$ such that for all Borel sets $A\subset V_i^{wu}$
\begin{equation}
 \label{eq:rho_FT_decay2}
 |\widehat{\chi\cdot\rho_A}(\xi)|\leq C_N \tu{vol}_{n_{wu}}(A)\langle \xi\rangle^{-N}
\end{equation}
uniformly for $\xi\in \Omega$. 
\end{prop}
\begin{proof}
In order to prove this theorem we use the definition of the Fourier transform 
for compactly supported distributions and get 
\begin{equation}
 \label{eq:FT_integral}
 |\widehat{\chi\cdot\rho_A}(\xi)| = |\rho_A[\chi(z)e^{-iz\xi}]| =\left| \int_{V_i} \mathds{1}_{S_A}(z) \chi(z)e^{-iz\xi}dz\right|.
\end{equation}
In order to obtain the estimates for this integral let us recall the following 
fact which is a consequence of the absolute continuity of the strong 
stable foliation.
\begin{thm}\label{thm:anosov_fubini}
For any $V_i\subset \R^n$ defined as above and any $x\in V_i^{wu}$ there is a 
smooth function $\delta_x \in C^\infty(S_x)$ such that for any 
$\psi\in C^\infty_c(V_i)$ 
\begin{equation}
 \label{eq:anosov_fubini}
 \int_{V_i} \psi(x,y) dxdy = \int_{V_i^{wu}} \left(\int_{S_x}\psi_{|S_x}(s)\cdot \delta_x(s) dm_{S_x}(s)\right) dx.
\end{equation}
Here $dx, dy$ are the euclidean measures on $\R^{n_{wu}},\R^{n_s}$, 
$\psi_{|S_x}\in C^\infty(S_x)$ is the restriction to the smooth submanifold $S_x$
and $dm_{S_x}$ the induced measure on the submanifold $S_x$.

The functions $\delta_x$ are called \emph{conditional densities} of the strong 
stable foliation. Furthermore, their  dependence on $x\in V_i^{wu}$ is continuous
in the following sense: If $T^1_s(S_x)$ is the set of normalized tangent vectors
 in $s\in S_x$, then for any $k\in \N$ the following $k$-norm 
 \[
  \|\delta_x\|_k := \sup_{s\in S_x}\sup_{X_1,\ldots,X_k\in T^1_s(S_x)}
  \left|X_1\dots X_k\delta_x(s)\right|
 \]
is finite and for all $k\in \N$ the map $x\to\|\delta_x\|_k$ is continuous.
\end{thm}
\begin{proof}
The existence of the conditional measures can be derived by a standard 
argument from the absolute continuity of the strong stable foliation 
(see Appendix~\ref{sec:appendix} respectively standard dynamical systems literature,
e.g. \cite[Section 6.2]{BS02}). In order to obtain the smoothness of the 
conditional measures as well as the continuous dependence on the base point 
$x$ one additionally needs the smoothness of the Jacobians of holonomy maps along
the strong stable leaves. This statement can for example be found in \cite[Appendix E]{GLP13}.
\end{proof}

First we use the Fubini formula for the strong stable foliation (\ref{eq:anosov_fubini}) 
in order to estimate the integral in (\ref{eq:FT_integral}) we obtain
\[
 |\widehat{\chi\cdot\rho_A}(\xi)| = \left|\int_{A} \left(\int_{S_x} \chi(s)e^{-is\xi} \delta_x(s)dm_{S_x}(s)\right)dx\right|.
\]
Now the uniform bounds on $\|\delta_x\|_0$ directly imply (\ref{eq:rho_FT_decay1}). 

In order to see (\ref{eq:rho_FT_decay2}) we use the boundedness of $\|\delta_x\|_k$ together
with the following partial integration argument: For $\xi_0 \in \Omega$ with 
$|\xi_0| =1$ and $x\in V_i^{wu}$ we define the tangent vector field
\[
 \mathcal X_{\xi_0,x}:\Abb{S_x}{TS_x\subset \R^n}{s}{\tu{Pr}_{T_sS_x}(\xi_0).} 
\]
Here $\tu{Pr}_{T_sS_x}$ is the orthogonal projection onto $T_sS_x$ which makes sense
after having identified $TV_i = V_i\times \R^n$. 
From the definition of $S_x$ and $\Gamma_2$ we conclude that for
any $x\in V_i^{wu}$ and $s\in S_x$ we have have 
$(s,\xi_0) \notin N^*S_x$. Thus there is $C_{\xi_0} > 0$ such that 
\[
\langle \mathcal X_{\xi_0,x}(s), \xi_0\rangle \geq C_{\xi_0}
\]
uniformly in $s\in S_x$ and $x\in V_i^{wu}$. By continuity we find an open 
neighborhood $W_{\xi_0}\subset \R^n$ of $\xi_0$ such that 
\[
\langle \mathcal X_{\xi_0,x}(s), \xi\rangle \geq \frac{1}{2} C_{\xi_0}
\]
uniformly in $\xi\in W_{\xi_0}$, $s\in S_x$ and $x\in V_i^{wu}$. 

Now we use the closedness of $\Omega$ and a compactness argument to conclude that
there are $\xi_1,\ldots,\xi_K$ as well as a constant $C_0 > 0$ such that 
for any $\xi \in \Omega$ there is $1\leq l\leq K$, such that
\begin{equation}
 \langle \mathcal X_{\xi_l,x}(s),\xi\rangle \geq C_0|\xi|
\end{equation}
uniformly in $x\in V_i^{wu}$, $s\in S_x$. Furthermore 
Theorem~\ref{thm:cont_depend_inv_mflds} implies that 
$s\mapsto \langle \mathcal X_{\xi_l,x}(s),\xi\rangle \in C^\infty(S_x)$. Even 
if this function is not compactly supported, the fact, that the $S_x$ are precompact,
and that these functions can always be smoothly extended to a slightly
larger stable leaf implies that for all $k\in \N$ the norm 
$\| \langle \mathcal X_{\xi_l,x}(s),\xi\rangle\|_k$ is finite and depends
continuously on $x\in V_i^{wu}$.

Now consider 
\[
 \left|\int_{S_x} \chi(s)e^{-is\xi}\delta_x(s) dm_{S_x}(s)\right|.
\]
Performing $N$ times partial integration w.r.t. the differential operator 
\[
 L_l:= \frac{1}{-i\langle \mathcal X_{\xi_l,x}(s),\xi\rangle}\mathcal X_{\xi_l,x}
\]
we get
\begin{eqnarray*}
 &&\left|\int_{S_x} \chi(s)e^{-is\xi}\delta_x(s) dm_{S_x}(s)\right| \\
 &\leq&  \int_{S_x} \left| 
 \left(\mathcal X_{\xi_l,x}\frac{1}{-i\langle \mathcal X_{\xi_l,x}(s),\xi\rangle}\right)^N \delta_x(s) \chi(s)\right|dm_{S_x}(s)\\
 &\leq & C_N \langle \xi \rangle^{-N},
\end{eqnarray*}
where $C_N$ is independent of $x\in V_i^{wu}$. Note that for this independence we 
crucially use that  $x\mapsto \|\delta_x\|_k$ as well as 
$x\mapsto\|\langle \mathcal X_{\xi_l,x}(s),\xi\rangle\|_k$ are continuous.
\end{proof}
Recall that we can choose $\Omega$ to be an arbitrary cone with $\Omega\cap\Gamma_2 = \{0\}$.
Thus (\ref{eq:rho_FT_decay2}) implies that for any Borel set $A\subset V_i^{wu}$ we have
$\rho_A \in \mathcal D'_{\Gamma_2}(V_i)$, i.e.~the wavefront set of $\rho_A$ is 
contained in $\Gamma_2$. Multiplication of $\rho_A$ with an arbitrary 
compactly supported distribution $v\in \mathcal E'_{\Gamma_1}(V_i)$ is therefore well 
defined and yields a compactly supported distribution $v\cdot\rho_A\in \mathcal E'(V_i)$.
\begin{lem}\label{lem:rho_A_prod_estimate}
For each distribution $v\in \mathcal E'_{\Gamma_1}(V_i)$ there is 
a constant $C_v>0$ such that for all Borel sets $A\subset V_i^{wu}$ we have 
\begin{equation}\label{eq:rho_A_prod_estimate}
 \left|(v\cdot \rho_A)[1] \right| \leq C_v\tu{vol}_{n_{wu}}(A).
\end{equation}
\end{lem}
\begin{proof}
 As $v$ has compact support in $V_i$ let us choose a cutoff function 
 $\chi\in C_c^\infty(V_i)$ with $\chi=1$ in a neighborhood of $\supp v$. Then we calculate
 \begin{align*}
  (v\cdot\rho_A)[1]& =(v\cdot(\chi\rho_A))[1]=\widehat{v\cdot(\chi\cdot \rho_A)}(0)=(\widehat v \star \widehat{\chi\cdot\rho_A})(0)\\
                   & =\int_{\R^n} \widehat v(\xi) \widehat{\chi\cdot\rho_A}(\xi)d\xi
 \end{align*}
and note that the last integral converges as $\Gamma_1\cap\Gamma_2=\{0\}$.

Now let us fix a closed cone $\Omega\subset \R^n$ fulfilling $\Gamma_1\setminus\{0\} \subset\tu{Int}\Omega$
and $\Omega\cap\Gamma_2=\{0\}$. As $\WF(v)\subset \Gamma_1$ we get for any $N\in \mathbb N$
a constant $C_{N}$ such that 
\[
|\hat v(\xi)|\leq C_{N,v}\langle\xi\rangle^{-N} 
\]
uniformly for $\xi\in \R^n\setminus \Omega$. As any distribution is of finite 
order we in addition get the existence of an integer $K_v$ and a constant $C_{v,2}$
such that 
\[
|\hat v(\xi)|\leq C_{v,2}\langle\xi\rangle^{K_v} 
\]
uniformly for all $\xi\in \R^n$. Splitting the integral over $\xi$ we thus obtain
\[
 |(v\cdot\rho_A)[1]|\leq C_{N,v} \int_{\R^n\setminus \Omega}  \langle\xi\rangle^{-N}\left|\widehat{\chi\cdot\rho_A}(\xi)\right|d\xi
 +C_{v,2} \int_{\Omega}  \langle\xi\rangle^{K_v}\left|\widehat{\chi\cdot\rho_A}(\xi)\right|d\xi
\]
for any given $N\in \N$. Then we can use (\ref{eq:rho_FT_decay1}) in the first 
integral and (\ref{eq:rho_FT_decay2}) with $N=K_v + n + 1$ in the second 
integral which leads to the desired estimate (\ref{eq:rho_A_prod_estimate}).
\end{proof}
Next we study the local structure of $\supp u\subset M$ for a resonant state 
$u\in \ker_{H_{sG}}(\mathbf P-\lambda_0)$. Therefore note that $\supp u \subset M$
is a $\varphi_t$-invariant closed subset. This is the only property that we 
use in the following lemma.
\begin{lem}\label{lem:full_measure_tubes}
Let $\Sigma\subset M$ be a closed subset invariant under the topological 
transitive Anosov flow. If $\Sigma \neq M$, then there exists a finite cover of 
product neighborhoods $M=\bigcup_{i=1}^N R_i$ as well as for any $i=1,\ldots,N$
an open subset $N_i\subset X_i$ such that 
\begin{equation}
 \label{eq:full_measure_tubes}
 \tu{vol}_{X_i}(X_i\setminus N_i) =0 \tu { and } 
  h_i(N_i \times Y_i) \cap \Sigma = \emptyset. 
\end{equation}
where $\tu{vol}_{X_i}$ is the Lebsegue measure on $X_i$.
\end{lem}
\begin{proof}
 As $M\setminus \Sigma$ is nonempty and open, there exists $\varepsilon_s, \varepsilon_{wu} >0$ and $m_0\in M$
 such that $\mathcal{PN}_{\varepsilon_{wu},\varepsilon_s}^{wu,s}(m_0) \cap \Sigma = \emptyset$. 
 Let us introduce the open sets 		
 \[
  U_0 := \mathcal{PN}_{\varepsilon_{wu},\varepsilon_s/2}^{wu,s}(m_0) ,~~U_1 := \mathcal{PN}_{\varepsilon_{wu},\varepsilon_s}^{wu,s}(m_0)
 \]
Now take a cover of product neighborhoods $M=\bigcup_{i=1}^N R_i$ with $r_i^s < \varepsilon_s /4$.
We now want to consider the backwards iterates $\varphi_{-t}(U_0)$ for $t>0$. Firstly,
the flow invariance of $\Sigma$ assures, that $\varphi_{-t}(U_0)\cap \Sigma=\emptyset$. 
Secondly the dynamical properties of the flow assures that the sets $\varphi_{-t}(U_0)$
are stretched into the strong stable direction and thirdly, if $\mu$ denotes the SRB 
measure, then ergodicity implies that $\mu(\bigcup_{t>0} \varphi_{-t}(U_0)) = \mu(M)$.
We will subsequently use these three facts in order to construct the sets $N_i$.

Therefore we first consider for an arbitrary index $i$ and $t>0$ the connected 
components of $\varphi_{-t}(U_0)\cap R_i$. If a connected component can be
written in the form $h_i(A\times Y_i)$ with $A\subset X_i$ open, we call it a
complete empty tube. We call them ``empty'' as these sets  have no intersection with 
$\Sigma$ and we intend to build our set $N_i$ from the sets $A$. Note that for large
$t$ the stretching and folding will imply that $\varphi_{-t}(U_0)\cap R_i$
has many connected components. But as we started with a product neighborhood $U_0$ and as the 
strong stable foliation is preserved by the flow, there are maximally two 
connected components which are not complete empty tubes. They correspond to 
the two ends of the product neighborhood (c.f. Figure~\ref{fig:tubes}).
We can however make them complete by considering the connected components
$\varphi_{-t}(U_1)\cap R_i$. As $U_1\cap \Sigma=\emptyset$, these connected components
are also empty. Furthermore the fact that $r_i^s<\varepsilon/4$ together with the
fact, that the product neighborhoods are stretched along the strong stable 
direction for negative times implies the following statement: Any connected 
component of $\varphi_{-t}(U_1)\cap R_i$ that contains a connected component 
of $\varphi_{-t}(U_1)\cap R_i$ is a complete tube. Thus we can define an open set
$A_{i,t}\subset X_i$ such that the union of connected 
components of $\varphi_{-t}(U_1)\cap R_i$ that contain a connected component 
of $\varphi_{-t}(U_0)\cap R_i$ is given by $h_i(A_{i,t}\times Y_i)$ with this 
definition we obviousely have
\[
 \varphi_{-t}(U_0)\cap R_i \subset h_i(A_{i,t}\times Y_i)\subset \varphi_{-t}(U_1)\cap R_i \subset M\setminus \Sigma.
\]
We now can define the opens set $N_i:= \bigcup_{t>0} A_{i,t}\subset X_i$
and obtain
\[
 \bigcup_{t>0} \varphi_{-t}(U_0)\cap R_i \subset h_i(N_i\times Y_i)\subset M\setminus \Sigma.
\]
As $h_i^{wu}:X_i \to V_i^{wu}$ is a diffeomorphism, it only remains to prove that 
$\tu{vol}_{n_{wu}}(h_i^{wu}(N_i))=\tu{vol}_{n_{wu}}(V_i^{wu})$.
This, however, is an immediate consequence of 
\[
 \mu(h_i(N_i\times Y_i)) \geq \mu(\bigcup_{t>0} \varphi_{-t}(U_0)\cap R_i) = \mu(R_i)
\]
(which uses the ergodicity w.r.t. the SRB measure) as well as the following Lemma.
\begin{figure}
\centering
        \includegraphics[width=1\textwidth]{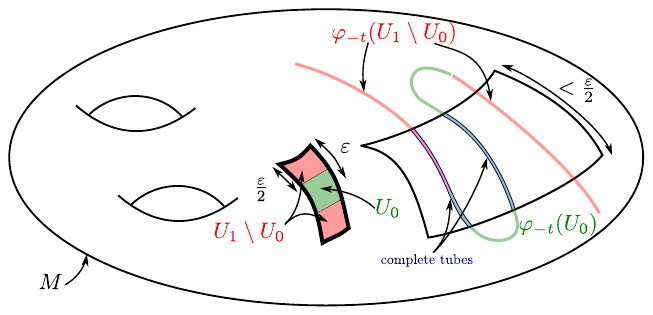}
\caption{Sketch visualizing the the choice of the complete strong stable tubes	}
\label{fig:tubes}
\end{figure}
\end{proof}
\begin{lem}\label{lem:neighbour_prod_measure}
 Let $\mu$ be the SRB measure, and $A \subset X_i$ with $\tu{vol}^{n_{wu}}(h_i^{wu}(A))>0$
 then $\mu(h_i(A\times Y_i)) > 0$.
\end{lem}
\begin{proof}
Using the fact that the SRB measure is absolutely continuous w.r.t. the weak
unstable foliation we can write 
\[
 \mu(h_i(A\times Y_i)) = \int_{B^s_{r_i^s(m)}}\underbrace{\left(\int_{h_i(A\times \{y\})} \tau_y\,dm_{\mathcal U_y}\right)}_{:=g(y)}d\sigma(y).
\]
Now, let us analyze the function $g(y)$. For $y=0$ we known from the definition 
of our charts, that $\mathcal U_0$ is diffeomorphic to $V_i^{wu}$. As $dm_{\mathcal U_0}$
comes from a Riemannian metric this induced measure is equivalent to Lebesgue 
measure on $V_i^{wu}$ and thus the assumption $\tu{vol}^{n_{wu}}(h_i^{wu}(A))>0$
implies $m_{\mathcal U_0}(h_i(A\times \{0\}))>0$. As $\rho_0$ is a positive density
this implies $g(0)>0$. Now let us consider an arbitrary other point $y\neq 0$. 
We want to show, that also the sets $h_i(A\times \{y\})\subset \mathcal U_y$ are 
of positive measure w.r.t. $m_{\mathcal U_y}$. This however is a direct consequence
of the fact that $h_i(A\times \{0\})$ has positive measure together with the fact,
that the strong stable holonomy maps are absolutely continuous (see e.g. 
\cite[Section 8.6.2]{BP07}). Again the positivity of the density $\rho_y$ implies
$g(y)>0$. $g$ is thus a strictly positive measurable function and consequently
$\mu(h_i(A\times Y_i)) = \int_{B^s_{r_i^s(m)}}g(y) d\sigma(y)>0$.
\end{proof}
We are now able to prove Theorem~\ref{thm:support}:
\begin{proof}
 Let $\lambda_0$ be a Pollicott-Ruelle resonance, 
 $\tilde u\in \ker_{H_{sG}}(\mathbf P-\lambda_0)^{J(\lambda_0)}\setminus\{0\}$ 
 an associated generalized resonant state and suppose that $\supp \tilde u \neq M$. 
 Let $k=0,\ldots, J(\lambda_0)-1$ be the unique integer, such that 
 \[
 u:=(\mathbf P-\lambda_0)^k\tilde u\in \ker_{H_{sG}}(\mathbf P-\lambda_0)\setminus\{0\}.
 \]
 As the differential operator $\mathbf P$ is local we have 
 $\supp u\subset \supp \tilde u\subsetneq M$. Furthermore because 
 $\varphi_{t}^*u = e^{-i\lambda_0} u$, the support $\supp u$ is invariant 
 by the Ansosov flow so we can apply Lemma~\ref{lem:full_measure_tubes} to 
 $\Sigma=\supp(u)$. Thus we obtain
 a cover $M=\bigcup_{i=1}^N R_i$ of product neighborhoods as well as open subsets 
 $N_i\subset X_i$ fullfilling
 \begin{equation}\label{eq:Ni_suppu}
  h_i(N_i\times Y_i)\cap \supp u = \emptyset.
 \end{equation}
 After a possible further refinement we can also 
 guarantee the existence of $C^\infty$-charts $\kappa_i:R_i\to V_i\subset \R^n$ as well as
 closed cones $\Gamma_1,\Gamma_2 \subset \R^n$ fulfilling (\ref{eq:un_stable_cones}).
 Now let us choose a smooth partition of unity 
 \[
 \sum_{i=1}^N \chi_i = 1\in C^\infty(M)\tu{ with }\chi_i\in C^\infty_c(R_i)
 \]
 and define $u_i:= (\kappa_i^{-1})^*(\chi_i\cdot u) \in \mathcal E'(V_i)$. Recall 
 from (\ref{eq:wavefront_cond}) that $u\in \ker_{H_{sG}}(\mathbf P-\lambda_0)$ implies 
 $\WF(u)\subset E^*_u$ and by (\ref{eq:un_stable_cones}) this implies that
 $\WF(u_i)\in \mathcal E'_{\Gamma_1}(V_i)$. Now, using Lemma~\ref{lem:rho_A_prod_estimate},
 we obtain for an arbitrary Borel set $A\subset V_i^{wu}$ the estimate 
 $|(u_i\cdot \rho_A)[1]|\leq C_{u_i}\tu{vol}_{n_{wu}}(A)$. We now want to use this 
 estimate in order to conclude, that $u_i = 0$ for all $i=1,\dots, N$. 

Now for any $\varepsilon>0$ let us introduce the set 
 \[
N_{i,\varepsilon}:=\{x\in N_i:~\{y\in \R^{n_{wu}}: |x-y|<\varepsilon\}\subset N_i\}  \subset V_i^{wu}.
 \]
As $N_i\subset V_i^{wu}$ are open, also the family $N_{i,\varepsilon}$ is a family of 
open sets, that increases monotonously and converges to $N_i$ for $\varepsilon\to 0$. Considering 
the complement $N_{i,\varepsilon}^c:=V_i^{wu}\setminus N_{i,\varepsilon}$ we obtain a 
monotonously decreasing family of closed sets converging to the nullset $V_i^{wu}\setminus N_i$.
Furthermore, from (\ref{eq:Ni_suppu}) we see that the distribution $\rho_{N^c_{i,\varepsilon}}$
defined in (\ref{eq:rho_A}) is equal 
to $1$ on an open neighborhood of $\supp(u_i)$. Taking an arbitrary 
$\psi\in C_c^\infty(V_i)$ and any $\varepsilon>0$ we can use this fact to calculate
\[
 u_i[\psi] = (u_i\cdot \psi)[1] = ((u_i\cdot\psi)\cdot \rho_{N^c_{i,\varepsilon}})[1].
\]
But as the multiplication by a smooth function doesn't increase the wavefront set 
we can apply Lemma~\ref{lem:rho_A_prod_estimate} to 
$(u_i\cdot \psi) \in \mathcal E'_{\Gamma_1}(V_i^{wu})$ and obtain for any $\varepsilon$
\[
 |u_i[\psi]| \leq C_{u_i,\psi} \tu{vol}_{n_{wu}}(N_{i,\varepsilon}^c).
\]
Finally, as the constants $C_{u_i,\psi}$ are independent of $\varepsilon$, the right
hand side converges to zero as $\varepsilon\to 0$, so we have shown $u_i[\psi]=0$. 
As $\psi\in C_c^\infty(V_i)$ was an arbitrary test function this implies $u_i=0$
for any $i$ and we conclude that $u=0$ which is a contradiction. This finishes 
the proof of Theorem~\ref{thm:support}.
\end{proof}
 \appendix
 \section{Smoothness of the conditional measures}\label{sec:appendix}
In the proof of the support property of the Pollicott-Ruelle resonant states,
the smoothness of the conditional measures (Theorem~\ref{thm:anosov_fubini})
plays a central role. While the result is well known to experts in dynamical system
theory it is difficult to find a precise statement or a proof in the existing 
literature. This is why I wanted to provide some more details in the arxiv-version 
of this article, on how to connect the smoothness of the conditional densities 
to the smoothness of holonomy maps, by standard dynamical systems arguments:

Let $V_i\subset \R^n$ be the range of a coordinate chart $\kappa_i$ as in 
Theorem~\ref{thm:anosov_fubini}. As we only work in one fixed chart, we drop
the index $i$ in order to simplify the notation and write $V:=V_i$. Recall that 
we have already defined $V^{wu} = \{x\in \R^{n_{wu}}: (x,0)\in V\}$ and analogousely 
we define $V^{s} :=\{y\in \R^{n_s}: (0,y)\in V\}$. Recall that the strong 
stable manifolds define a foliation of $V$ by the leaves $S_x$ which are parametrized
by $x\in V^{wu}$. Note that $S_0 = V\cap(\{0\}\times \R^{n_s})$ so it can be 
identified with $V^s$. Now for any $y\in V^s$
we define $T_y:= V\cap(\R^{n_{wu}}\times\{y\})$ which defines a smooth foliation
transversal to the foliation $S_x$. Recall that by the construction of the 
coordinate charts (\ref{eq:chart_properties}), $T_0$ coincides with a weak unstable leaf, but in general the
other leaves can not coincide with the weak unstable leaves unless 
the particular case, that the weak unstable foliation is smooth. Let us now reduce
the set $V$ such that it has a product structure w.r.t. the foliations $S_x$ and $T_y$,
i.e. such that for all $x\in V^{wu}$ and $y\in V^s$ one has $S_x\cap T_y =:\mathcal H(x,y)$
defines a unique point contained in $V$. Note that this poses not problem for
proving the theorem which claims the existences of the conditional measures in 
some product neighborhood of the strong stable and weak unstable foliation as such a set is 
always contained in the reduced set $V$. 

Now with this product structure of $S_x$ and $T_y$ we can define two holonomy maps
\[
 \mathcal H^S_y:\Abb{T_0=V^{wu}}{T_y}{x}{\mathcal H(x,y)}
\]
and 
\[
\mathcal H^T_x:\Abb{S_0=V^{s}}{S_x}{y}{\mathcal H(x,y)}. 
\]
Let us first discuss the second one, belonging to the smooth transversal foliation.
According to Theorem~\ref{thm:cont_depend_inv_mflds}, $\mathcal H^T_x$
is a $C^\infty$ diffeomorphism and depends continuously on $x$ in the 
$C^\infty$-topology. Note that the euclidean metric on $V$ induces metrics and thus measures 
on the submanifolds $S_x$ which we call $dm_{S_x}(s)$ and as $S_0$ is simply the 
$y$-axis $dm_{S_0}$ is simply the euclidean measure $dy$ on $V^s$ after identifying
$S_0$ and $V^s$. 
The fact that $\mathcal H^T_x$ are diffeomorphisms gives us a Jacobian function 
$J^T_x\in C(S_0)$ such that for a any function $\psi\in C_c(S_x)$
\[
 \int_{S_x} \psi(s) dm_{S_x}(s) = \int_{S_0} \psi\circ \mathcal H_x^T(y) J_x^T(y) dm_{S_0}(y).
\]
The theorem of change of variables gives us, the precise form of $J_x^T(y)$: 
For any $y\in S_0$ the differential of the holonomy maps is a linear
isomorphism
\[
 D\mathcal H_x^T: T_y S_0\subset \R^n \to T_y S_0\subset \R^n.
\]
Choosing two orthonormal basis of the $n_s$-dimensional subspaces 
$T_y S_0\subset \R^n$ and $T_y S_0\subset \R^n$ (w.r.t.~to the euclidean
metric on $\R^n$) the differential can be expressed as an $n_s\times n_s$
matrix and the Jacobian is simply the absolute value of the determinant. Now the fact, 
that $\mathcal H_x^T$ is $C^\infty$ and depends continuously on $x$ translates to
the fact that $J^T_x\in C^\infty(S_0)$ and depends continuously on $x$ w.r.t.
the $C^\infty$-topology. 

Let us now turn to the holonomy maps of the strong stable foliation. Here the 
situation is infinitely more complicated, as for a general Anosov flow the maps 
$\mathcal H^S_y$ are only Hölder continuous homeomorphisms. However it is known 
that they are absolutely continuous and that they possess a Jacobian $J_y^S$ such 
that for any function $\psi\in C_c(T_y)$
\[
 \int_{T_y} \psi(x) d\tu{Leb}_{T_y}(x) = \int_{T_0} \psi\circ \mathcal H_y^S(x) J_y^S(x) d\tu{Leb}_{T_0}(x).
\]
The existence of such Jacobians for holonomies in dynamical systems can be found in 
various textbooks (see e.g. \cite[Section 6.2]{BS02}. It is however more difficult
to find sharp statements on their regularity in $x$ and $y$. The following
sufficiently strong statement for our purpose can for example be found in 
\cite[Appendix E]{GLP13}
\begin{prop}
For any $x$, the map $y\in S_0 \mapsto J_y^S(x)\in \R_{>0}$
is in $C^\infty(S_0)$ and the map $x\in V^{wu}\mapsto J_\bullet^S(x)\in C^\infty(S_0)$ 
is continuous \footnote{Note that in \cite[Appendix E]{GLP13} they state the result 
for the strong unstable foliation. Furthermore they give even more precise information
on the Hölder regularity of the map $x\mapsto J_\bullet^S(x)$
which is not relevant for our purpose.}
\end{prop}
Let us now combine all the ingredients in order to calculate the conditional 
densities. Given a function $\psi\in C_c(V)$ we calculate
\begin{eqnarray*}
 \int_V \psi(x,y) dx\,dy&=& \int_{V^s}\left[\int_{T_y} f(x,y) d\tu{Leb}_{T_y}(x)\right]dy\\
 &=&\int_{V^s}\left[ \int_{T_0} \psi\circ \mathcal H_y^S(x) J_y^S(x) d\tu{Leb}_{T_0}(x)\right]dy\\
 &=&\int_{T_0}\left[\int_{S_0} \psi(\mathcal H(x,y)) J_y^S(x) dm_{S_0}(y)\right]d\tu{Leb}_{T_0}(x)\\
  &=&\int_{T_0}\left[\int_{S_x} \psi(s) \frac{J^S_{(\mathcal H^T_x)^{-1}(s)}(x)}{J^T_x((\mathcal H^T_x)^{-1}(s))} dm_{S_x}(s)\right]d\tu{Leb}_{T_0}(x).
\end{eqnarray*}
We have thus shown that the conditional densities take the form 
\[
 \delta_x(s) =\frac{J^S_{(\mathcal H^T_x)^{-1}(s)}(x)}{J^T_x((\mathcal H^T_x)^{-1}(s))}.
\]
Collecting the statements on the regularity of the Jacobians and holonomy maps we
conclude, that for any $x$ they are smooth and that the map
\[
 x\mapsto \delta_x(H_x^T(\bullet)) \in C^\infty(S_0)
\]
depends continuously on $x$ w.r.t. the $C^\infty$ topology. This statement 
assures, the uniform bounds on $\|\delta_x\|_k$.

\begin{rem}
 In the case of $C^\infty$-foliations we could take coordinate charts $\kappa_i$ such that
 $S_x = (\{x\}\times \R^{n_s})\cap V_i$. Then (\ref{eq:anosov_fubini}) would reduce to the ordinary 
 Fubini theorem. For a general foliation such generalizations of Fubini theorems are a nontrivial
 task and there are counterexamples of foliations that are not absolutely continuous (see e.g. 
 \cite[Section 8.6]{BP07}).
\end{rem}
\begin{rem}
For Anosov maps, an analogouse statement to Theorem~\ref{thm:anosov_fubini} can be 
found in \cite[Remark 3.4, p.377]{Che02}.
\end{rem}


\providecommand{\bysame}{\leavevmode\hbox to3em{\hrulefill}\thinspace}
\providecommand{\MR}{\relax\ifhmode\unskip\space\fi MR }
\providecommand{\MRhref}[2]{%
  \href{http://www.ams.org/mathscinet-getitem?mr=#1}{#2}
}
\providecommand{\href}[2]{#2}

\end{document}